\documentclass[leqno]{amsart}
\usepackage{caption}
\usepackage{amsmath,amssymb,amsthm,mathtools,dsfont,mathrsfs,tikz}
\usepackage{environ}
\usepackage{stmaryrd}
\usetikzlibrary{arrows}
\usetikzlibrary{positioning}
\usetikzlibrary{decorations.text}
\usetikzlibrary{decorations.pathmorphing}
\usetikzlibrary{decorations.pathreplacing}
\makeatletter
\newsavebox{\measure@tikzpicture}
\NewEnviron{scaletikzpicturetowidth}[1]{%
  \def\tikz@width{#1}%
  \def\tikzscale{1}\begin{lrbox}{\measure@tikzpicture}%
  \BODY
  \end{lrbox}%
  \pgfmathparse{#1/\wd\measure@tikzpicture}%
  \edef\tikzscale{\pgfmathresult}%
  \BODY
}
\makeatother
\usepackage{bbm}
\DeclarePairedDelimiter\norm{\lvert}{\rvert}
\DeclarePairedDelimiter\inner{\langle}{\rangle}

\usepackage{bm}
\usepackage[colorlinks=true,linkcolor=blue,citecolor=blue]{hyperref}
\usepackage{enumitem}
\usepackage{csquotes}

\numberwithin{equation}{section}
\newcounter{intro}

		\newtheorem{introthm}[intro]{Theorem}

		\newtheorem{thm}[equation]{Theorem}
		\newtheorem{lem}[equation]{Lemma}
		\newtheorem{cor}[equation]{Corollary}

\theoremstyle{remark}
		\newtheorem{rem}[equation]{Remark}
\theoremstyle{definition}
		\newtheorem{defn}[equation]{Definition}
		\newtheorem{exam}[equation]{Example}

\title{A {J}ordan--{H}\"{o}lder Type Theorem for Supercharacter Theories}
\author{Shawn T. Burkett}
\address{Department of Mathematical Sciences, Kent State University, Kent,
Ohio 44240, U.S.A.} \email{sburket1@kent.edu}
\date{\today}
\begin{document}
\maketitle
\begin{abstract}
The Jordan--H\"{o}lder Theorem is a general term given to a collection of theorems about maximal chains in suitably nice lattices. For example, the well-known Jordan--H\"{o}lder type theorem for chief series of finite groups has been rather useful in studying the structure of finite groups. In this paper, we present a Jordan--H\"{o}lder type theorem for supercharacter theories of finite groups, which generalizes the one for chief series of finite groups.
\end{abstract}
\section{Introduction}
The Jordan--H\"{o}lder Theorem for finite groups serves as a kind of unique factorization theorem, where we may think of the group as a being a \enquote{generalized} product of its simple composition factors. This fundamental result, originally proved by C. Jordan, and  strengthened to its current statement by O. H\"{o}lder (see the remark following Theorem 5.12 in \cite{JJRgps}), is a generalization of another well-known result of R. Dedekind that states that the multiset of chief factors of any chief series is an invariant of the group. In this way, the group $G$ is a generalized product of its chief factors, and if one could classify every such product, one could classify all groups with a given chief series. This group extension problem is a generally hard problem, and still an active area of research.

Another generally hard problem for a finite group is the classification of its set of supercharacters theories. A supercharacter theory $\mathsf{S}$ of $G$ is, roughly speaking, an approximation of its complex character theory, where the irreducible characters of $G$ are replaced by certain pairwise orthogonal characters, called $\mathsf{S}$-irreducible characters, constant the parts of a suitable partition of $G$. The appeal of supercharacter theory is that that the relationship between the $\mathsf{S}$-irreducible characters and $\mathsf{S}$-classes largely mimics that of the irreducible characters and conjugacy classes, and supercharacter theories may be constructed in situations where the full character theory is difficult or intractable to describe. In this paper we give a Jordan--H\"{o}lder type theorem for supercharacter theories which recovers the classical theorem for chief series. In particular, this theorem will serve as a type of unique factorization theorem for supercharacter theories, encoding not only group theoretic information but also supercharacter theoretic. To achieve this, we work with a particular subset of normal subgroups which are \enquote{seen} by the supercharacter theory $\mathsf{S}$.

A subgroup which arises as an intersection of kernels of $\mathsf{S}$-irreducible characters is called $\mathsf{S}$-normal (or supernormal). It turns out (see \cite[Lemma 3.6]{SB18nil}) that we may construct a sublattice of $\mathrm {Norm}(G)$ from any a supercharacter theory $\mathsf{S}$ of a finite group $G$ by considering only the collection $\mathrm{Norm}(\mathsf{S})$ of $\mathsf{S}$-normal subgroups of a supercharacter theory $\mathsf{S}$ of $G$. In his Ph.D. thesis \cite{AH08}, A. Hendrickson introduced the concept of $\mathsf{S}$-normality, and shows that there is a supercharacter theory $\mathsf{S}_{N/H}$ induced on $N/H$ whenever $H\le N$ are $\mathsf{S}$-normal subgroups. Hendrickson also shows that whenever $N$ is an $\mathsf{S}$-normal subgroup of $G$, that $\mathsf{S}$ is related in a predictable way to another supercharacter theory $\mathsf{S}_N\ast\mathsf{S}_{G/N}$, a supercharacter theory built from the the supercharacter theories $\mathsf{S}_N$ and $\mathsf{S}_{G/N}$. Specifically every $\mathsf{S}_N\ast\mathsf{S}_{G/N}$ is the union of some $\mathsf{S}$-classes. So in some sense, $\mathsf{S}$ and $\mathsf{S}_N\ast\mathsf{S}_{G/N}$ are both generalized products of $\mathsf{S}_N$ and $\mathsf{S}_{G/N}$, where the latter retains only some information about $\mathsf{S}$. Expanding the connection between a supercharacter theory and its supernormal subgroups, Aliniaeifard shows in \cite{FA17} that given any sublattice $L$ of $\mathrm{Norm}(G)$, there is a supercharacter theory $\mathsf{A}(L)$ whose set of supernormal subgroup is exactly $L$. Moreover, if $\mathsf{S}$ is any other supercharacter theory satisfying $\mathrm{Norm}(\mathsf{S})=L$, then every $\mathsf{A}(L)$-class is a union of $\mathsf{S}$-classes. However, for every minimal quotient $H/N$ in $L$, one has $\mathsf{A}(L)_{H/N}$ is the trivial supercharacter theory of $H/N$; i.e., its superclasses are $\{1\}$ and $H/N\setminus\{1\}$. So although this supercharacter theory shares the same lattice of supernormal subgroups as $\mathsf{S}$, it does not (in general) share the \enquote{local} supercharacter theoretic information.

The main ingredient used in a standard proof of the Jordan--H\"{o}lder Theorem, and its many variations, including the just mentioned earlier, is the modular law. Since the lattice $\mathrm {Norm}(G)$ of normal subgroups of a finite group is modular, so too will be any sublattice (a subset which is a lattice under the same meet and join); in particular $\mathrm{Norm}(\mathsf{S})$ is modular. It can then be shown that every maximal chain in $\mathrm{Norm}(\mathsf{S})$ has the same length, and the isomorphism classes of subquotients is independent of the choice of maximal chain. We call a maximal chain in $\mathrm{Norm}(\mathsf{S})$ an $\mathsf{S}$-{\bf chief series}. In particular, any two $\mathsf{S}$-chief series have the same length, and the subquotients of any $\mathsf{S}$-chief series are isomorphic up to permutation. However, much more can be said. To be more precise, we need to first give a definition. We say that the supercharacter theories $\mathsf{S}$ of $G$ and $\mathsf{T}$ of $H$ are {\bf isomorphic via} $\varphi$, where $\varphi:G\to H$ is an isomorphism, if the $\mathsf{T}$-characters of $H$ have the form $\chi\circ\varphi^{-1}$, where $\chi$ is an $\mathsf{S}$-irreducible character of $G$.

\begin{introthm}\label{diamond}
Let $\mathsf{S}$ be a supercharacter theory of $G$, and suppose that $H$ and $N$ are $\mathsf{S}$-normal. Let $\varphi$ be the isomorphism $H/(H\cap N)\to HN/N$, $h(H\cap N)\mapsto hN$. The supercharacter theories $\mathsf{S}_{H/(H\cap N)}$ and $\mathsf{S}_{HN/N}$ are isomorphic via $\varphi$.
\end{introthm}

Theorem~\ref{diamond}, reminiscent of the second isomorphism theorem for groups, shows that the structure of the supercharacter theories induced on $\mathsf{S}$-normal subquotients are not only determined by $\mathsf{S}$, but from other quotients. This relatively easy result has proven itself quite useful. In fact, it is used heavily to prove the main result of the paper.
\begin{introthm}\label{JH2}
Let $\mathsf{S}$ be a supercharacter theory of $G$. Let $G=N_1>N_2>\dotsb>N_s=1$ and $G=H_1>H_2>\dotsb>H_s=1$ be two $\mathsf{S}$-chief series. There exists a permutation $\tau$ of $\{1,2,\dotsc,s-1\}$ and isomorphisms $\varphi_i:N_i/N_{i+1}\to H_{\tau(i)}/H_{\tau(i)+1}$ such that $\mathsf{S}_{N_i/N_{i+1}}\simeq_{\varphi_i}\mathsf{S}_{H_{\tau(i)}/H_{\tau(i)+1}}$ for each $1\le i\le s-1$.
\end{introthm}

As a consequence of this result, one may not only think of $G$ as being a generalized product of its $\mathsf{S}$-chief factors, but also of $\mathsf{S}$ as being a generalized product of their induced theories. 

We conclude this section by briefly discussing a potential application of Theorems~\ref{diamond} and \ref{JH2}. As mentioned above, Aliniaeifard gives a method of constructing a supercharacter theory $\mathsf{A}(L)$ from any sublattice $L$ of $\mathrm{Norm}(G)$. One may then consider the possibility of taking a sublattice $L$ of $\mathrm{Norm}(G)$ along with supercharacter theories of the minimal subquotients of $L$ to build a supercharacter theory finer than $\mathsf{A}(L)$ which retains not only the same set of supernormal subgroups, but also the same supercharacter theoretic information at the minimal subquotients. Theorems~\ref{diamond} and ~\ref{JH2} place necessary restrictions on the choices of supercharacter theories placed on the subquotients.


This work was part of the author's Ph.D. thesis at the University of Colorado Boulder under the supervision of Nathaniel Thiem. The author expresses his gratitude to Dr. Thiem for the many lengthy and insightful conversations, and helpful advice that led to this work. The author is also grateful for the numerous suggestions and valuable advice of Mark L. Lewis during the preparation of this manuscript.

\section{Supercharacter theories and supernormality}
In this section, we review the basics of supercharacter theory and supernormality. As defined in \cite{ID07}, a {\bf supercharacter theory} $\mathsf{S}$ of a group $G$ is a pair $(\mathcal{X}_{\mathsf{S}},\mathcal{K}_{\mathsf{S}})$, where $\mathcal{X}_{\mathsf{S}}$ is a partition of $\mathrm{Irr}(G)$, $\mathcal{K}_{\mathsf{S}}$ is a partition of $G$ containing $\{1\}$ satisfying the conditions
\begin{itemize}
\item$\norm{\mathcal{K}_{\mathsf{S}}}=\norm{\mathcal{X}_{\mathsf{S}}}$;
\item for each $X\in\mathcal{X}_{\mathsf{S}}$, there exists a character $\xi_X$ such that $\mathrm{Irr}(\xi_X)\subseteq X$, and $\xi_X$ is constant on the parts of $\mathcal{K}_{\mathsf{S}}$.
\end{itemize}
The characters $\xi_X$ can be taken to be the characters $\sigma_X=\sum_{\psi\in X}\psi(1)\psi$, and $\{\mathbbm{1}\}\in\mathcal{X}_{\mathsf{S}}$. For each $X\in\mathcal{X}_{\mathsf{S}}$, we call $\sigma_X$ an $\mathsf{S}$-{\bf irreducible} character, and we denote the set of all $\mathsf{S}$-irreducible characters by $\mathrm{Irr}(\mathsf{S})$.
 The parts of $\mathcal{K}_{\mathsf{S}}$ are unions of $G$-conjugacy classes; we call them $\mathsf{S}$-{\bf classes} and let $\mathrm{Cl}(\mathsf{S})$ denote the set of $\mathsf{S}$-classes.

The irreducible characters and conjugacy classes of $G$ both determine a special subset of subgroups of $G$ --- the normal subgroups. Indeed every normal subgroup appears as the intersection of the kernels of some collection of irreducible characters, and such a subgroup is the union of some conjugacy classes. For a supercharacter theory $\mathsf{S}$, an $\mathsf{S}$-{\bf normal} subgroup is a subgroup which is a union of $\mathsf{S}$-classes. Such a subgroup also arises as the intersection of some $\mathsf{S}$-irreducible characters. It is clear that $\mathsf{S}$-normal subgroups are also normal. If $N$ is an $\mathsf{S}$-normal subgroup of $G$, we may write $N\lhd_{\mathsf{S}}G$.

In his Ph.D. thesis \cite{AH08}, and in the subsequent paper \cite{AH12}, Hendrickson illustrates the importance of $\mathsf{S}$-normal subgroups in supercharacter theory. He shows that one may naturally associate a {\bf restricted} supercharacter theory $\mathsf{S}_N$ of $N$ whenever $N$ is $\mathsf{S}$-normal. The $\mathsf{S}_N$-classes are just the $\mathsf{S}$-classes contained in $N$, and the $\mathsf{S}_N$-irreducible characters are (up to a scalar) restrictions of the $\mathsf{S}$-irreducible characters. He also defines a {\bf deflated} supercharacter theory $\mathsf{S}^{G/N}$ of $G/N$ for $N\lhd_{\mathsf{S}}G$; the $\mathsf{S}^{G/N}$-irreducible characters are $\mathsf{S}$-irreducible characters that contain $N$ in their kernel (naturally considered characters of $G/N$), and the $\mathsf{S}^{G/N}$-classes are the images of the $\mathsf{S}$-classes under the canonical surjection $G\to G/N$.

It is not difficult to show that these constructions are compatible in the sense that
\[(\mathsf{S}^{G/H})_{N/H}=(\mathsf{S}_N)^{N/H}\]
whenever $H\le N$ are $\mathsf{S}$-normal subgroups of $G$. We will therefore write, unambiguously, $\mathsf{S}_{N/H}$ to denote the supercharacter theory induced on the subquotient $N/H$ from $\mathsf{S}$.

Given a normal subgroup $N$ of $G$ and supercharacter theories $\mathsf{S}$ of $N$ and $\mathsf{T}$ of $G/N$, one obtains a supercharacter theory $\mathsf{S}\ast\mathsf{T}$ of $G$, as long as the $\mathsf{S}$-classes are invariant under the conjugation action of $G$. This supercharacter theory is called the $\ast$-product (see \cite{AH12} for details) of $\mathsf{S}$ and $\mathsf{T}$ and has supercharacters
\[\mathrm{Irr}(\mathsf{S}\ast\mathsf{T})=\bigl\{\mathrm{Ind}_N^G(\psi):\mathbbm{1}\neq\psi\in\mathrm{Irr}(\mathsf{S})\bigr\}\cup\mathrm{Irr}(\mathsf{T})\]
and superclasses
\[\mathrm{Cl}(\mathsf{S}\ast\mathsf{T})=\mathrm{Cl}(\mathsf{S})\cup\bigl\{\pi^{-1}(K):N\neq K\in\mathrm{Cl}(\mathsf{T})\bigr\},\]
where the characters of $G/N$ are naturally identified with characters of $G$, and $\pi:G\to G/N$ is the canonical projection.
\section{Schur rings}
A supercharacter theory is very closely related to another algebraic structure called a Schur ring. The theory of Schur rings was primarily developed by Issai Schur and Helmut Wielandt for studying permutation groups, especially primitive permutation groups. We refer the reader to \cite{schur1933,wielandt1950,HW64} for details. 

\begin{defn}
Let $G$ be a finite group, and let $\mathscr{S}$ be a unital subalgebra of $\mathbb{C}{G}$ under the {\bf ordinary} product
\[\sum_{g\in G}a_gg\cdot\sum_{g\in G}b_gg=\sum_{g,h\in G}a_gb_hgh.\] Let $\mathcal{K}$ be a partition of $G$ and let $\widehat{K}=\sum_{g\in K}g$. \label{sums}Then $\mathscr{S}$ is called a {\bf Schur ring} over $G$ and $\mathcal{K}$ its corresponding {\bf Schur partition} if the following conditions hold:
\begin{itemize}
\item[{\normalfont\bf(1)}] $\mathscr{S}=\mathbb{C}$-$\mathrm{span}\bigl\{\widehat{K}: K\in\mathcal{K}\bigr\}$.
\item[{\normalfont\bf(2)}] $K^{-1}\coloneqq\{g^{-1}: g \in K\}\in\mathcal{K}$ for each $K\in\mathcal{K}$.
\end{itemize}
\end{defn}

The Schur ring perspective essentially allows one to study a supercharacter theory by only considering the supercharacters, or by only considering the superclasses. In some situations, it is easy to describe one and not the other, so the Schur ring perspective is also quite useful in determining if a partition of $\mathrm{Irr}(G)$ (or of $G$) gives rise to the supercharacters (superclasses) of a supercharacter theory of $G$. In this section, we outline many of the basic properties of Schur rings, as well as their connection to supercharacter theories.

The connection between supercharacter theories and Schur rings was first noticed by Hendrickson in \cite{AH12}, where the following theorem appears.
\begin{thm}{\normalfont \cite[Proposition 2.4]{AH12}}\label{bijection}\hspace{\labelsep}Let $G$ be a finite group. The function 
\[\mathsf{S}\mapsto\mathbb{C}\text{-}\mathrm{span}\bigl\{\widehat{K}: K\in\mathrm{Cl}(\mathsf{S})\}\] 
is a bijection
\[\left\{\begin{array}{c}\text{\normalfont supercharacter theories}\\\text{\normalfont of $G$}\end{array}\right\}\longrightarrow\left\{\begin{array}{c}\text{\normalfont central Schur rings}\\\text{\normalfont  over $G$}\end{array}\right\}.\]
\end{thm}

The above bijection says that the superclass sums form a basis for a Schur ring. However, it makes no mention of the role that supercharacters play in Schur ring theory. Although it appears to be somewhat unknown in the supercharacter theory community, an analog of character theory for Schur rings was developed in the 1960s by Tamaschke in the papers \cite{OT69,OT70} which captures all of ingredients of supercharacter theory. Since Tamaschke does not limit his attention to central (or unital) Schur rings, actually quite a lot more is done than what lies within the scope of supercharacter theory. In the paper \cite{OT70}, Tamashcke exploits semisimplicity and uses representation theoretic arguments to construct the characters afforded by the simple $\mathcal{A}$-modules of a Schur ring $\mathcal{A}$. His landmark paper could be viewed as the first paper on supercharacter theory, although this vernacular was not used, of course. 

Viewing $\mathcal{A}$ as a supercharacter theory by the bijection of Theorem~\ref{bijection}, the supercharacters of $\mathcal{A}$ also form an orthogonal basis for $\mathrm{cf}(\mathcal{A})$ by \cite[Theorem 2.2 (a)]{ID07}. We use the modern perspective of supercharacters to illustrate the connection of character theory to Schur rings via idempotents of subalgebras of $\mathrm{cf}(G)$. 

Note that the group algebra $\mathbb{C}{G}$ under the ordinary product is isomorphic to the algebra $\mathbb{C}^G$ of functions $G\to\mathbb{C}$
with the convolution product $\ast$ given by
\[(\alpha\ast\beta)(g)=\sum_{x\in  G}\alpha(xg)\beta(x^{-1});\]
the isomorphism $\Theta$ given by
\[\begin{array}{rl}\Theta:\mathbb{C}{G}&\longrightarrow\;\;\mathbb{C}^G\\\sum_{g\in G}a_gg&\longmapsto\;\raisebox{-2ex}{$\begin{array}{rl}\alpha_g: G&\longrightarrow\;\mathbb{C}\\g&\longmapsto\;\, a_g.\end{array}$}\end{array}\]

\begin{thm}{\normalfont \cite[Theorem 2.18]{MI76}}\label{gencol}\hspace{\labelsep}
Let $\chi,\psi\in\mathrm{Irr}(G)$. Then
\[\frac{1}{\norm{G}}(\chi\ast\psi)(g)=\delta_{\chi,\psi}\frac{\chi(g)}{\chi(1)}.\]
\end{thm}
By Theorem~\ref{gencol}, often referred to as {\it generalized column orthogonality}, we see that the functions $\chi(1)\chi/\norm{G}$ for $\chi\in\mathrm{Irr}(G)$ are orthogonal idempotents for the convolution product. Since $\mathrm{Irr}(G)$ is a basis of $\mathrm{cf}(G)$, this means that $\{\chi(1)\chi/\norm{G}:\chi\in\mathrm{Irr}(G)\}$ is an orthogonal idempotent basis for $\mathrm{cf}(G)$ as an algebra under the convolution product.

Recall that the set $\{e_\psi:\psi\in\mathrm{Irr}(G)\}$, where $e_\psi=\frac{\psi(1)}{\norm{G}}\sum_{g\in G}\psi(g^{-1})g$, is an idempotent basis for $\mathbf{Z}(\mathbb{C}(G)$. Let $\mathcal{A}$ be a central Schur ring over $G$. Since $\mathcal{A}$ is a subalgebra of the commutative semisimple algebra $\mathbf{Z}(\mathbb{C}{G})$, $\mathcal{A}$ is also semisimple. So $\mathcal{A}$ has a basis of orthogonal idempotents. As noted above, an idempotent basis of $\mathbf{Z}(\mathbb{C}{G})$ with respect to the ordinary product is $\{e_\psi:\psi\in\mathrm{Irr}(G)\}$.
Since $1=\sum_\psi e_{\psi}$, there must be some partition $\mathcal{X}$ of $\mathrm{Irr}(G)$ so that $\{e_X: X\in\mathcal{X}\}$ is an idempotent basis for $\mathcal{A}$, where $e_X=\sum_{\psi\in X}e_\psi$. Since $\Theta(e_X)=\sigma_X/\norm{G}$, the set of functions $\sigma_X/\norm{G}$ for $X\in\mathcal{X}$ give an orthogonal idempotent basis for the subalgebra $\Theta(\mathcal{A})$ of $\mathrm{cf}(G)$, with respect to the convolution product. 

The group algebra is an algebra under another associative product $\circ$ called the {Hadamard} product, defined by the rule
\[\sum_{g\in G}a_gg\circ\sum_{g\in G}b_gg=\sum_{g\in G}a_gb_gg.\]
The following result states the role of the Hadamard product in Schur ring theory. This result is essentially stated in \cite{MPschur} with the added condition of being closed under inverses; however, Hendrickson noticed in \cite{AH12} that one does not need to include this condition when considering only central subalgebras. To tie these together, we include a proof.
\begin{thm}[{\normalfont cf. \cite{MPschur,AH12}}]\label{schur} 
Let $\mathcal{A}$ be a unital subalgebra of $\mathbf{Z}(\mathbb{C}{G})$ with respect to the ordinary product containing $\widehat{G}$. Then $\mathcal{A}$ is a Schur ring if and only if $\mathcal{A}$ is closed under the Hadamard product.
\end{thm}
\begin{proof}
We have already observed that $\mathcal{A}$ is semisimple with respect to the ordinary product. Since $\mathcal{A}$ is closed under $\circ$, a similar argument shows that $\mathcal{A}$ is also semisimple with respect to $\circ$. Thus, we may compute idempotent bases with respect to both products. Above, we showed that there is some partition $\mathcal{X}$ of $\mathrm{Irr}(G)$ so that $\{e_X:X\in\mathcal{X}\}$ is an orthogonal idempotent basis for $\mathcal{A}$. Analogously, there is some set $\mathcal{K}$ of mutually disjoint subsets of $G$ so that $\{\widehat{K}: K\in\mathcal{K}\}$ is an idempotent basis for $\mathcal{A}$ with respect to $\circ$. Since $\widehat{G}\in\mathcal{A}$, $\mathcal{K}$ must be actually be a partition of $G$.

Now, 
\[e_X=\sum_{\psi\in X}e_\psi=\frac{1}{\norm{G}}\sum_{g\in G}\sum_{\psi\in X}\psi(1)\overline{\psi(g)}g=\frac{1}{\norm{G}}\sum_{g\in G}\overline{\sigma_X(g)}g=\sum_{K\in\mathcal{K}}a_K\widehat{K},\]
for some coefficients $a_K\in\mathbb{C}$. Thus, it must be the case that $\sigma_X$ is constant on the parts of $\mathcal{K}$ for each $X\in\mathcal{X}$. The map $g\mapsto g^{-1}$ therefore permutes the parts of $\mathcal{K}$; also, $\{1\}\in\mathcal{K}$ by \cite[Theorem 2.2]{ID07}. Hence $\mathcal{A}$ is a Schur ring. 

Now let $\mathcal{A}$ be a Schur ring with corresponding Schur partition $\mathcal{K}$. Since $K\cap K'=\varnothing$ if $K\neq K'$, we have $\widehat{K}\circ\widehat{K'}=\delta_{KK'}\widehat{K}$. It follows by linearity that $\mathcal{A}$ is closed under $\circ$.
\end{proof}

Observe that under the isomorphism $\Theta$, the algebra $\mathbb{C}{G}$ with the Hadamard product is isomorphic to the algebra $\mathbb{C}^G$ with the pointwise product $\cdot$ defined by \[(\alpha\cdot\beta)(g)=\alpha(g)\beta(g).\] Therefore, if $\mathcal{A}$ is a Schur ring, then $\Theta(\mathcal{A})$ is a subalgebra of $\mathrm{cf}(G)$ with respect to both the convolution product and the pointwise product. Similarly, if $\mathsf{S}$ is a subalgebra of $\mathrm{cf}(G)$ with respect to both the convolution and pointwise products, then $\Theta^{-1}(\mathsf{S})$ is a subalgebra of $\mathbf{Z}(\mathbb{C}{G})$ with respect to the ordinary product and the Hadamard product, and is thus a Schur ring by Proposition~\ref{schur}. Since $\Theta(\widehat{G})=\mathbbm{1}$, we have the following corollary to Theorem~\ref{schur}, which was also proved by Andrews in \cite{SAlittle}.

\begin{cor}[{\normalfont \cite[Lemma 2.2]{SAlittle}}]\label{SA14}
Let $\mathsf{S}$ be a unital subalgebra of $\mathrm{cf}(G)$ with respect to the convolution product containing $\{\mathbbm{1}\}$. Then $\mathsf{S}=\mathrm{cf}(\mathcal{A})$ for a Schur ring $\mathcal{A}$ if and only if and $\mathsf{S}$ is closed under the pointwise product.
\end{cor}
\begin{figure}[h!]
\vspace{1.5cm}
\centering
\mbox{\fontsize{8}{8}\selectfont
\begin{scaletikzpicturetowidth}{.95\textwidth}
\begin{tikzpicture}[scale=\tikzscale]
\draw (1.25,0) node[left] {$\left\{\begin{array}{c}\text{supercharacter}\\\text{theories of $G$}\end{array}\right\}$};

\draw [->,thick,postaction={decorate,decoration={text along path,text align=center,text={{$\mathsf{S}$}{$\,\mapsto\,$}{$\mathbb{C}$}-span{$\{$}{$\widehat{K}$}{$:$}{$K$}{$\in$}{Cl(}{$\mathsf{S}$}{)}{$\}$}{}},raise=1.5ex}}]      (-.5,1.15) to [bend left=45] (4,1.15);

\draw [->,thick,postaction={decorate,decoration={text along path,text align=center, text={{$\begin{array}{c}\sum\limits_{g\in G}\\\mbox{}\end{array}$}{$\mspace{-18mu}\begin{array}{c}a_g\\\mbox{}\end{array}$}{$\mspace{-18mu}\begin{array}{c}g\\\mbox{}\end{array}$}{$\mspace{-15mu}\begin{array}{c}\mapsto\\\mbox{}\end{array}$}{$\mspace{-15mu}\begin{array}{c}\alpha_g\\\mbox{}\end{array}$}{$\mspace{-18mu}\begin{array}{c}\,:\;\\\mbox{}\end{array}$}{$\mspace{-15mu}\begin{array}{c}G\\g\end{array}$}{$\mspace{-15mu}\begin{array}{c}\to\\\mapsto\end{array}$}{$\mspace{-15mu}\begin{array}{c}\mathbb{C}\\a_g\end{array}$}{}},raise=3ex}}] (6,1.15) to [bend left=45]  (10.5,1.15);

\draw (5,0) node{$\left\{\begin{array}{c}\text{unital subalgebras of}\\\text{$\mathbf{Z}(\mathbb{C}{G})$ closed under}\\\text{ $\circ$, containing $\widehat{G}$}\end{array}\right\}$};
\draw (10.75,0) node{$\left\{\begin{array}{c}\text{unital subalgebras of}\\\text{$\mathrm{cf}(G)$ closed under}\\\text{$\cdot$, containing $\mathbbm{1}$}\end{array}\right\}$};
    
\draw [<-,thick,postaction={decorate,decoration={text along path,text align=center, text={{$\sum\limits_{g\in G}$}{$\alpha(g)$}{$g$}{$\,\mapsfrom$}{$\,\alpha$}{}},raise=-3ex}}] (6,-1.15) to [bend right=45]  (10.5,-1.15);
    
 \draw [<-,thick,postaction={decorate,decoration={text along path,text align=center,text={express basis in the{}},raise=-3ex}}]      (-.5,-1.15) to [bend right=45] (4,-1.15);
 
  \draw [decoration={text along path,   text={{$\{e_\chi\}$}{ and }{$\{\widehat{K}\}$} bases{}},text align={center},raise=-6ex},decorate]
    (-.5,-1.15) to [bend right=45] (4,-1.15);
\end{tikzpicture}
\end{scaletikzpicturetowidth}
}\\[4ex]
\caption*{Supercharacter theories and Schur rings}\label{sctring}
\end{figure}

\section{$\mathsf{S}$-chief series and the Jordan--H\"{o}lder Theorem}

In this section, we prove our main theorem. We first develop several more results involving supernormal subgroups, and discuss an analog of chief series. 

\begin{lem}[{\normalfont\cite[Lemma 3.6]{SB18nil}}]\label{sublattice}Let $\mathsf{S}$ be a supercharacter theory of $G$, and suppose that $H$ and $N$ are $\mathsf{S}$-normal. Then $H\cap N$ and $HN$ are both $\mathsf{S}$-normal.
\end{lem}
As an immediate corollary, we have tha following.
\begin{cor}The set $\mathrm{Norm}(\mathsf{S})$ of all $\mathsf{S}$-normal subgroups forms a sublattice of the lattice of all normal subgroups of $G$. 
\end{cor}
The lattice $\mathrm{Norm}(\mathsf{S})$ has structure beyond the level of group theory. In order to discuss this, we will require some notion of isomorphism of supercharacter theories.
\begin{defn}\label{sctiso}Let $\varphi:G\to H$ be a group isomorphism. Let $\mathsf{S}$ be a supercharacter theory of $G$, and let $\mathsf{T}$ be a supercharacter theory of $H$. We will say that $\mathsf{S}$ and $\mathsf{T}$ are {\bf isomorphic via} $\varphi$ and write $\mathsf{S}\simeq_\varphi\mathsf{T}$ if 
\[\mathrm{Cl}(\mathsf{T})=\{\varphi(K):K\in\mathrm{Cl}(\mathsf{S})\}.\]
\end{defn}
As introduced earlier, the previous definition can be equivalently defined via characters. 
\begin{lem}
Let $\varphi:G\to H$ be a group isomorphism. The supercharacter theory $\mathsf{S}$ of $G$ and the supercharacter theory $\mathsf{T}$ of $H$ are isomorphic via $\varphi$ if and only if
\[\mathrm{Irr}(\mathsf{T})=\{\chi\circ\varphi^{-1}:\chi\in\mathrm{Irr}(\mathsf{S})\}.\]
\end{lem}
\begin{proof}
This follows from the fact that the partitions $\{\mathrm{Irr}(\chi\circ\varphi^{-1}):\chi\in\mathrm{Irr}(\mathsf{S})\}$ of $\mathrm{Irr}(H)$, and $\{\varphi(K):K\in\mathrm{Cl}(\mathsf{S})\}$ of $H$ satisfy the conditions of a supercharacter theory, and that superclasses and supercharacters uniquely determine each other \cite[Theorem 2.2(c)]{ID07}.
\end{proof}

\begin{rem}
A group isomorphism $G\to H$ induces isomorphism $\mathbb{C}{G}\to\mathbb{C}{H}$ of the corresponding group algebras. The restriction of this isomorphism to a Schur ring over $G$ gives a Schur ring over $H$, and the two Schur rings are said to Cayley isomorphic. When restricting one's attention only to central Schur rings over $G$, Cayley isomorphic agrees with the concept of isomorphism of supercharacter theories given by A. Lang in \cite{langsemi}. In particular, the choice of the word \enquote{isomorphic} in Definition~\ref{sctiso} is appropriate.
\end{rem}

We now prove Theorem~\ref{diamond}, which shows that the structure of the induced theories is controlled in some sense. 
\begin{proof}[Proof of Theorem~\ref{diamond}]
Let $\pi_1:G\to G/(H\cap N)$ and $\pi_2:G\to G/N$. Then $\pi_2(H)=(\varphi\circ\pi_1)(H)$. Since the classes of $\mathsf{S}_{HN/N}$ are of form $\pi_2(K)$ for $K\in\mathrm{Cl}(\mathsf{S})$ and the classes of $\mathsf{S}_{H/(H\cap N)}$ are of form $\pi_1(K)$ for $K\in\mathrm{Cl}(\mathsf{S})$ with $K\subseteq H$, the result follows.
\end{proof}

In the situation of Theorem~\ref{diamond}, we simply write $\mathsf{S}_{H/(H\cap N)}\simeq\mathsf{S}_{HN/N}$, instead of $\mathsf{S}_{H/(H\cap N)}\simeq_\varphi\mathsf{S}_{HN/N}$.

\begin{lem}[{\normalfont\cite[Lemma 2.1]{SB18nil}}]\label{quotient}
Let $N$ be $\mathsf{S}$-normal in $G$, and let $H\le G$ contain $N$. Then $H$ is $\mathsf{S}$-normal if and only if $H/N$ is $\mathsf{S}_{G/N}$-normal.
\end{lem}

We will now describe an analog of chief series for supercharacter theories. Given a supercharacter theory $\mathsf{S}$ of $G$ and a $\mathsf{S}$-normal subgroup $N$ of $G$, Lemma~\ref{quotient} shows that any supernormal subgroup of $G/N$ yields an $\mathsf{S}$-normal subgroup of $G$. So for $G/N$ to be an analog of a chief factor, we should require that $\mathsf{S}_{G/N}$ have no nontrivial proper supernormal subgroups. We will say that a supercharacter theory $\mathsf{S}$ of $G$ is {\bf simple} if the only $\mathsf{S}$-normal subgroups of $G$ are $1$ and $G$. 

\begin{defn}
Let $\mathsf{S}$ be a supercharacter theory of $G$. A series $G=N_1>N_2>\dotsb>N_r=1$ is called an $\mathsf{S}$-{\bf chief} series if $N_i\lhd_{\mathsf{S}}G$ and $\mathsf{S}_{N_{i}/N_{i+1}}$ is simple for each $1\le i\le r-1$.
\end{defn}
As a result of Lemma~\ref{sublattice}, we have that the set of $\mathsf{S}$-normal subgroups of $G$ form a sublattice of the lattice of all normal subgroups of $G$. Since the latter is modular, so is the lattice of $\mathsf{S}$-normal subgroups, which we will denote by $\mathrm{Norm}(\mathsf{S})$. In particular, the standard Jordan--H\"{o}lder Theorem applies.
\begin{thm}[Jordan--H\"{o}lder]\label{JH1}Let $\mathsf{S}$ be a supercharacter theory of $G$. Let $G=N_1>N_2>\dotsb>N_r=1$ and $G=H_1>H_2>\dotsb>H_s=1$ be $\mathsf{S}$-chief series of $G$. Then $r=s$, and there exists a permutation $\tau$ of $\{1,2,\dotsc,s-1\}$ such that $N_i/N_{i+1}$ is isomorphic to $H_{\tau(i)}/H_{\tau(i)+1}$ for each $1\le i\le s-1$.
\end{thm}
A Jordan--H\"{o}lder Theorem for general (even nonunital) Schur rings was discovered by O. Tamaschke \cite[Theorem 10.3]{OT69}. This was done by developing the underlying category theory of Schur rings. By considering subgroups which essentially generalize subnormal subgroups and (what we call) supernormal subgroups, Tamashcke gave a Schur ring theoretic version of the Zassenhaus (Butterfly) Lemma, and of the  Schreier Refinement Theorem. In the event that a Schur ring is central, as is the case with a supercharacter theory, Tamaschke's theorem is related to Theorem~\ref{JH1}. However, Tamaschke's theorem doesn't guarantee the subquotients are isomorphic up to permutation, but instead that they have the same order and that certain Schur rings related to the subquotients are isomorphic. Theorem~\ref{diamond} allows us to prove Theorem~\ref{JH2}, a Jordan--H\"{o}lder type theorem specifically for supercharacter theories which is stronger than both Tamaschke's result and Theorem~\ref{JH1} in the case of $\mathsf{S}$-chief series. We restate Theorem~\ref{JH2} here.

\setcounter{intro}{1}
\begin{introthm}Let $\mathsf{S}$ be a supercharacter theory of $G$. Let $G=N_1>N_2>\dotsb>N_s=1$ and $G=H_1>H_2>\dotsb>H_s=1$ be two $\mathsf{S}$-chief series, which necessarily have the same length by Proposition~\ref{JH1}. There exists a permutation $\tau$ of $\{1,2,\dotsc,s-1\}$ and isomorphisms $\varphi_i:N_i/N_{i+1}\to H_{\tau(i)}/H_{\tau(i)+1}$ such that $\mathsf{S}_{N_i/N_{i+1}}\simeq_{\varphi_i}\mathsf{S}_{H_{\tau(i)}/H_{\tau(i)+1}}$ for each $1\le i\le s-1$.
\end{introthm}
\begin{proof}
Assume the result is false, and choose $G$ to have minimal order among all groups possessing a supercharacter theory that is a counterexample to the theorem. If $s=2$, then clearly $\mathsf{S}$ is not a counterexample to the theorem, so suppose that $s=3$. If $N_2=H_2$, then the result is trivial so assume that $N_2\neq H_2$. Then $G=N_2H_2$ and $N_2\cap H_2=1$, so $G/H_2\simeq N_2$ and $G/N_2\simeq H_2$. By Theorem~\ref{diamond}, $\mathsf{S}$ is also not a counterexample to the Theorem. So let $\mathsf{S}$ be a supercharacter theory of $G$ for which the result fails, and note that we must have $s\ge 4$. The result holds trivially if $\mathrm{Norm}(\mathsf{S})$ has only one maximal chain, so we may assume that this is not the case. 

Assume that $N_j=H_j$ for some $3\le j\le s-2$. Then the two series $N_j>N_{j+1}>\dotsb>N_s=1$, and $N_j=H_j>H_{j+1}>\dotsb>H_s=1$ are both  $\mathsf{S}_{N_j}$-chief series, so, by the minimality of $G$, there exists a permutation $\tau$ of $\{j,j+1,\dotsc,s-1\}$ and isomorphisms $\varphi_i:N_i/N_{i+1}\to H_{\tau(i)}/H_{\tau(i)+1}$, for each $j\le i\le s-1$, such that  
\begin{align*}\mathsf{S}_{N_i/N_{i+1}}\simeq_{\varphi_i}\mathsf{S}_{H_{\tau(i)}/H_{\tau(i)+1}}\end{align*}
for each $j\le i\le s-1$. Similarly, by considering the supercharacter theory $\mathsf{S}_{G/N_j}$ of $G/N_j$, we obtain by minimality for some permutation $\tilde{\tau}$ of $\{1,2,\dotsc,j-1\}$ and isomorphisms $\varphi_i$ such that 
\begin{align*}\mathsf{S}_{N_i/N_{i+1}}\simeq_{\varphi_i}\mathsf{S}_{H_{\tilde{\tau}(i)}/H_{\tilde{\tau}(i)+1}}\end{align*}
for $1\le i\le j-1$. In particular, by taking $\rho$ to be the permutation of $\{1,2,\dotsc,s-1\}$ that acts by $\tilde{\tau}$ on $\{1,2,\dotsc,j-1\}$ and by $\tau$ on $\{j,j+1,\dotsc,s-1\}$, we have a permutation $\rho$ of $\{1,2,\dotsc,s-1\}$ and isomorphisms $\varphi_i:N_i/N_{i+1}\to H_{\rho(i)}/H_{\rho(i)+1}$ such that $\mathsf{S}_{N_i/N_{i+1}}\simeq_{\varphi_i}\mathsf{S}_{H_{\rho(i)}/H_{\rho(i)+1}}$ for each $1\le i\le s-1$. This is a contradiction, by the choice of $\mathsf{S}$.

We may now assume that there is no $j$ with $3\le j\le s-2$ such that $N_j=H_j$. Define $B_3=N_2\cap H_2$, $C_4=B_3\cap N_3$ and $D_4=B_3\cap H_3$. Since $N_2H_2=G$, it follows that
\[G/N_2\simeq H_2/B_3\ \,\text{and}\ \,G/H_2\simeq N_2/B_3\]
so $B_3$ is maximal in both $N_2$ and $H_2$. Similarly $C_4$ is maximal in $B_3$ and $N_3$ and $D_4$ is maximal in $B_3$ and $H_3$. Choose $B_i$, $C_i$ and $D_i$ so that we have the following $\mathsf{S}$-chief series:
\begin{align*}
{\bf(1)}\hspace{2\labelsep}&G=N_1>N_2>N_3>N_4>\dotsb>N_s=1,\\
{\bf(2)}\hspace{2\labelsep}&G=N_1>N_2>N_3>C_4>\dotsb>C_s=1,\\
{\bf(3)}\hspace{2\labelsep}&G=N_1>N_2>B_3>C_4>\dotsb>C_s=1,\\\
{\bf(4)}\hspace{2\labelsep}&G=H_1>H_2>B_3>D_4>\dotsb>D_s=1,\\\
{\bf(5)}\hspace{2\labelsep}&G=H_1>H_2>H_3>D_4>\dotsb>D_s=1,\\
{\bf(6)}\hspace{2\labelsep}&G=H_1>H_2>H_3>H_4>\dotsb>H_s=1.
\end{align*}
We may assume that $N_3>1$ since this case was covered earlier. Now, if $N_4$ is trivial, then so are $B_4$, $C_4$ and $D_4$ and $H_4$ by Proposition~\ref{JH1}. Now, $N_3$ and $B_3$ are maximal in $N_2$ and since $1=C_4=N_3\cap B_3$, the series {\bf(1)}--{\bf(6)} become
\begin{align*}
{\bf(1)}\hspace{2\labelsep}&G>N_2>N_3>1,\\
{\bf(2)}\hspace{2\labelsep}&G>N_2>N_3>1,\\
{\bf(3)}\hspace{2\labelsep}&G>N_2>B_3>1,\\
{\bf(4)}\hspace{2\labelsep}&G>H_2>B_3>1,\\
{\bf(5)}\hspace{2\labelsep}&G>H_2>H_3>1,\\
{\bf(6)}\hspace{2\labelsep}&G>H_2>H_3>1.
\end{align*}
\noindent By Theorem~\ref{diamond}, we have
\[\mathsf{S}_{N_2/N_3}\simeq\mathsf{S}_{B_3}\ \,\text{and}\ \,\mathsf{S}_{N_3}\simeq\mathsf{S}_{N_2/B_3}.\]

Similarly, we have
\[\mathsf{S}_{H_2/H_3}\simeq\mathsf{S}_{B_3}\ \,\text{and}\ \,\mathsf{S}_{H_3}\simeq\mathsf{S}_{H_2/B_3}\]
and \[\mathsf{S}_{N_2/B_3}\simeq\mathsf{S}_{G/H_2}\ \,\text{and}\ \,\mathsf{S}_{G/N_2}\simeq\mathsf{S}_{H_2/B_3}.\]
These isomorphisms show that $\mathsf{S}$ is not a counterexample to the theorem, so $N_4$ must be nontrivial. 

If $N_5$ is trivial, then so are $B_5$, $C_5$ and $D_5$ and $H_5$ by Proposition~\ref{JH1}. Then the $\mathsf{S}$-chief series {\bf(1)}--{\bf(6)} become

\begin{align*}
{\bf(1)}\hspace{2\labelsep}&G>N_2>N_3>N_4>1,\\
{\bf(2)}\hspace{2\labelsep}&G>N_2>N_3>C_4>1,\\
{\bf(3)}\hspace{2\labelsep}&G>N_2>B_3>C_4>1,\\
{\bf(4)}\hspace{2\labelsep}&G>H_2>B_3>D_4>1,\\
{\bf(5)}\hspace{2\labelsep}&G>H_2>H_3>D_4>1,\\
{\bf(6)}\hspace{2\labelsep}&G>H_2>H_3>H_4>1.
\end{align*}

\noindent Theorem~\ref{diamond} can be used repeatedly as in the last case to obtain the chains of isomorphisms

\begin{align*}
\mathsf{S}_{G/N_2}&\simeq\mathsf{S}_{H_2/B_3}\simeq\mathsf{S}_{H_3/D_4}\simeq\mathsf{S}_{H_4},\\
\mathsf{S}_{N_2/N_3}&\simeq\mathsf{S}_{B_3/C_4}\simeq\mathsf{S}_{D_4}\simeq\mathsf{S}_{H_3/H_4},\\
\mathsf{S}_{N_3/N_4}&\simeq\mathsf{S}_{C_4}\simeq\mathsf{S}_{B_3/D_4}\simeq\mathsf{S}_{H_2/H_3},\\
\shortintertext{and}
\mathsf{S}_{N_4}&\simeq\mathsf{S}_{N_3/C_4}\simeq\mathsf{S}_{N_2/B_3}\simeq\mathsf{S}_{G/H_2}.
\end{align*}
By the choice of $\mathsf{S}$, we must therefore have $N_5>1$. 

Since $N_5>1$, we have $s\ge 6$. Hence, we are reduced to the first case for each pair of chains {\bf(i)} and {\bf(i+1)}, $1\le i\le 5$. Composing these permutations and isomorphisms, we reach a final contradiction.\end{proof}

\begin{exam}
Let $G=\inner{x,y:x^3=y^6=[x,y]=1}$ be the group $C_3\times C_6$. One may easily verify that the set
\[\bigl\{\{1\},\{y,y^3,y^5\},\{y^2,y^4\},\{xy^2\},\{x^2y^4\},\{x,xy^4\},\{x^2,x^2y^2\},K_1,K_2\bigr\},\]
where $K_i=\{x^iy,x^iy^3,x^iy^5\}$ for $i=1,2$, gives the superclasses for a supercharacter theory $\mathsf{S}$ of $G$. The set of $\mathsf{S}$-normal subgroups of $G$ is $\{1,\inner{xy^2},\inner{y^2},\inner{y},\inner{x,y^2},G\}$. Therefore, the lattice of $\mathsf{S}$-normal subgroups is given by the following.

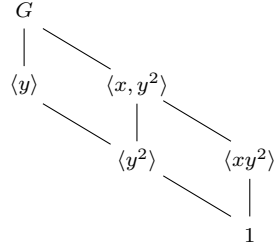
\begin{center}
\begin{tikzpicture}
\draw (1.5,0) node {\footnotesize$1$};
\draw (1.5,1) node {\footnotesize$\inner{xy^2}$};
\draw (0,1) node {\footnotesize$\inner{y^2}$};
\draw (0,2) node {\footnotesize$\inner{x,y^2}$};
\draw (-1.5,2) node {\footnotesize$\inner{y}$};
\draw (-1.5,3) node {\footnotesize$G$};
	\draw (1.5,.25)--(1.5,.75);
	\draw (0,1.25)--(0,1.75);
	\draw (-1.5,2.25)--(-1.5,2.75);
\draw (1.25,.25)--(.35,.79);
\draw (1.25,1.25)--(.35,1.79);
\draw (-.35,1.21)--(-1.25,1.75);
\draw (-.35,2.21)--(-1.25,2.75);
\end{tikzpicture}
\end{center}
Observe that
 \begin{align*}\mathsf{S}_{\inner{xy^2}}&=\bigl\{\{1\},\{xy^2\},\{x^2y^4\}\bigr\},
 \shortintertext{and} 
\mathsf{S}_{\inner{x,y^2}/\inner{y^2}}&=\bigl\{\{\inner{y^2}\},\{x\inner{y^2}\},\{x^2\inner{y^2}\}\bigr\}.
\end{align*} 
These supercharacter theories both coincide with the finest supercharacter theory of $C_3$, and they are all isomorphic via the canonical map $\inner{xy^2}\to\inner{x,y^2}/\inner{y^2}$. We also have
\[\mathsf{S}_{G/\inner{y}}=\bigl\{\{\inner{y}\},\{x\inner{y}\},\{x^2\inner{y}\}\bigr\},\]
which is isomorphic to $\mathsf{S}_{\inner{xy^2}}$ via the canonical map $\inner{xy^2}\to G/\inner{y}$ and to $\mathsf{S}_{\inner{x,y^2}/\inner{y^2}}$ via the canonical map $\inner{x,y^2}/\inner{y^2}\to G/\inner{y}$
Next, we note that 
\begin{align*}
\mathsf{S}_{G/\inner{x,y^2}}&=\bigl\{\{\inner{x,y^2}\},\{y\inner{x,y^2}\}\bigr\}\\
\shortintertext{and} 
\mathsf{S}_{\inner{y}/\inner{y^2}}&=\bigl\{\{\inner{y^2}\},\{y\inner{y^2}\}\bigr\}, 
\end{align*}
and that  these supercharacter theories coincide with the finest (only) supercharacter theory of $C_2$, and they are also isomorphic via the canonical map $\inner{y}/\inner{y^2}\to G/\inner{x,y^2}$. Now observe that
\begin{align*}
\mathsf{S}_{\inner{y^2}}&=\bigl\{\{1\},\{y^2,y^4\}\bigr\}\\
\shortintertext{and} 
\mathsf{S}_{\inner{x,y^2}/\inner{xy^2}}&=\bigl\{\{\inner{xy^2}\},\{x\inner{xy^2},xy^4\inner{xy^2}\}\bigr\}\\
&=\bigl\{\{\inner{xy^2}\},\{y^4\inner{xy^2},y^2\inner{xy^2}\}\bigr\}.
\end{align*}
Therefore, we have these two supercharacter theories are isomorphic via the canonical map $\inner{y^2}\to\inner{x,y^2}/\inner{xy^2}$.

Finally we note that the only three $\mathsf{S}$-chief series are the following:
\begin{align*}
G&\ge \inner{x,y^2}\ge\inner{xy^2}\ge1,\\
G&\ge \inner{x,y^2}\ge\inner{y^2}\ge1,\\
\shortintertext{and}
G&\ge \inner{y}\ge\inner{y^2}\ge1.\\
\end{align*}
The above observations give us the permutations and isomorphisms guaranteed by the theorem.
\end{exam}
\bibliographystyle{alpha}
\bibliography{bio}
\end{document}